\documentclass{article}
\usepackage{amsfonts}
\usepackage{amsmath}

\newtheorem{theorem}{Theorem}

\newtheorem{definition}[theorem]{Definition}
\newtheorem{example}[theorem]{Example}

\newtheorem{proposition}[theorem]{Proposition}
\newtheorem{remark}[theorem]{Remark}

\newenvironment{proof}[1][Proof]{\noindent\textbf{#1.} }{\ \rule{0.5em}{0.5em}}
\input{tcilatex}

\begin{document}

\title{Quadratic Forms and their Number Systems and Geometry}
\author{\.{I}skender \"{O}ZT\"{U}RK \\
ORC\.{I}D : 0000-0001-5674-8219, iskender.ozturk06@gmail.com}
\maketitle

\begin{abstract}
$\ $Quatenions, split quaternions and hybrid number are very well known
number systems. These number systems are used to make geometry in Euclidean
and Lorentz spaces. These number systems can be obtained with the help of a
quadratic form. This article examines how to find the corresponding number
system for any quadratic form. As a result of this examination, bilinear
form, vector product, skew symmetric matrix, and rotation matrices
corresponding to any number system were obtained.


\end{abstract}

\section{Introduction}

Quadratic forms are a widely studied topic. Complex, hyperbolic, and dual
numbers are well-known two-dimensional number systems. By associating these
number systems with quadratic forms, the geometric structure of these number
systems has been examined. Rotation transformations on the unit circle,
hyperbola and dual circle in $\mathbb{R}^{2}$ are examined in detail in the
article \cite{harkin2004}. In this study, rotational transformation on
three-dimensional objects such as ellipsoid, hyperboloid and cone in $%
\mathbb{R}^{3}$ will be examined. For this, firstly, a number system of any
given quadratic form will be created. Then, using this number system, a
rotational transformation will be performed on the quadric surface
corresponding to a quadratic form, just like the effect of quaternions on
the sphere and the effect of split quaternions on the hyperboloid. Now,
let's introduce the general properties of quadratic forms and their forms in 
$\mathbb{R}^{3}$.

\begin{definition}
\cite{parimala2013} A quadratic form $q:\mathcal{V}\longrightarrow F$ on a
finite dimentional vector space $\mathcal{V}$ over $F$ is a map satisfying:

i\textit{.} $Q\left( \lambda \mathbf{v}\right) =\lambda ^{2}Q\left( \mathbf{v%
}\right) $ for $\mathbf{v}\in $ $\mathcal{V},\lambda \in F.$

ii. The map%
\begin{equation}
\mathcal{B}_{Q}\left( \mathbf{v,w}\right) =\dfrac{1}{2}\left[ Q\left( 
\mathbf{v}+\mathbf{w}\right) -Q\left( \mathbf{v}\right) -Q\left( \mathbf{w}%
\right) \right]  \label{identity properties}
\end{equation}%
is bilinear.\newline
We denote a quadratic form by $\left( \mathcal{V},Q\right) ,$ or simply as $%
Q.$ The bilinear form $\mathcal{B}_{Q}$ is symmetric; $Q$ determines $%
\mathcal{B}_{Q}$ and for all $\mathbf{v}\in \mathcal{V},$%
\[
Q\left( \mathbf{v}\right) =\mathcal{B}_{Q}\left( \mathbf{v,v}\right)
\]%
where $F$ is a field with \textit{char}$F\neq 2$.
\end{definition}

\begin{remark}
Since $F=\mathbb{R}$ will be taken in this article, the character of the
field $F$ will not be mentioned in the rest of the article.
\end{remark}

Given a bilinear form on $\mathbb{R}^{3}$, there exists a unique $I^{\ast
}\in \mathbb{R}^{3\times 3}$ square matrix such that for all $\mathbf{u},%
\mathbf{v}\in \mathbb{R}^{3}$, $\mathcal{B}_{Q}\left( \mathbf{u,v}\right) =%
\mathbf{u}I^{\ast }\mathbf{v}$. $I^{\ast }$ is called "\textit{the matrix
associated with the form}" with respect to the standard basis. A bilinear
form is said skew-symmetric if $\mathcal{B}_{Q}\left( \mathbf{u,v}\right) =-%
\mathcal{B}_{Q}\left( \mathbf{v,u}\right) ,$ respectively. Also, the matrix
associated with a symmetric bilinear form is symmetric, and similarly, the
associated matrix of a skew-symmetric bilinear form is skew-symmetric.

Due to (\ref{identity properties}) symmetric bilinear forms and quadratic
forms are in a one-to-one correspondence over $\mathbb{R}$. The kernel of a
symmetric bilinear form $\mathcal{B}$ on $\mathcal{V}$ is 
\begin{equation}
\ker \mathcal{B}=\left\{ \mathbf{v}\in \mathcal{V}:\mathcal{B}\left( \mathbf{%
v,w}\right) =0\text{ for all }\mathbf{w}\in \mathcal{V}\right\} .
\label{kernel}
\end{equation}%
As can be seen from the definition (\ref{kernel}) kernel of a bilinear form
is a linear subspace of $\mathcal{V}$. The kernel of a quadratic form is
defined as the kernel of corresponding bilinear form. A symmetric bilinear
form $\mathcal{B}$ is called%
\[
\begin{array}{ll}
\text{degenerate,} & \text{if }\ker \mathcal{B}\neq \left\{ 0\right\} , \\ 
\text{non-degenerate,} & \text{if }\ker \mathcal{B}=\left\{ 0\right\} .%
\end{array}%
\]

A symmetric bilinear form is degenerate if and only if its matrix with
respect to one basis has determinant $0$. A quadratic form is called
degenerate or nondegenerate if the corresponding bilinear form is degenerate
or nondegenerate. The vector space $\mathbb{R}^{3}$ over $\mathbb{R}$, a
bilinear form $\mathcal{B}$ (and the corresponding quadratic form) is called%
\[
\begin{array}{ll}
\text{positive semidefinite,} & \text{if }\mathcal{B}_{Q}\left( \mathbf{v,v}%
\right) \geq 0\text{, for }\forall \mathbf{v}\in \mathbb{R}^{3} \\ 
\text{positive definite} & \text{if }\mathcal{B}_{Q}\left( \mathbf{v,v}%
\right) >0\text{, for }\forall \mathbf{v}\in \mathbb{R}^{3}-\left\{ 0\right\}
\\ 
\text{indefinite} & \text{if }\mathcal{B}_{Q}\left( \mathbf{v,v}\right) <0%
\text{, }\mathcal{B}_{Q}\left( \mathbf{w,w}\right) >0\text{, for }\exists 
\mathbf{v}\text{,}\mathbf{w}\in \mathbb{R}^{3}.%
\end{array}%
\]

The conditions negative semidefinite and negative definite are defined
similarly. But, a real non-degenerate symmetric bilinear form is either
positive definite, negative definite, or indefinite. The character of a
bilinear form as positive or negative definite (or neither) can also be
determined by the determinant test. But these details will not be discussed
here.

A scalar (bilinear) product on a real vector space $\mathcal{V}$ is a
non-degenerate symmetric bilinear form. We will represent the scalar product
as $\left\langle \cdot ,\cdot \right\rangle .$ A vector space equipped with
a scalar product is called an inner product space and it is denoted as $(%
\mathcal{V},\left\langle \cdot ,\cdot \right\rangle )$. For any vectors $%
\mathbf{u,v}$ in an inner product space are called orthogonal if $%
\left\langle \mathbf{u},\mathbf{v}\right\rangle =0$, also if $\left\langle 
\mathbf{u},\mathbf{u}\right\rangle =\left\langle \mathbf{v},\mathbf{v}%
\right\rangle =1$, then these vectors are called orthonormal.

The Euclidean scalar product is a scalar product that is positive definite.
The subject of this article the vector space $\mathbb{R}^{3}$ equipped with
a Euclidean scalar product is called a Euclidean vector space. The Euclidean
norm of a vector $\mathbf{x}$ is%
\[
\left\Vert \mathbf{x}\right\Vert =\sqrt{\left\langle \mathbf{x},\mathbf{x}%
\right\rangle _{\mathbb{E}}}.
\]%
A Lorentz scalar product is an indefinite scalar product. The vector space $%
\mathbb{R}^{3}$ equipped with a Lorentz scalar product with signature $%
\left( -,+,+\right) $ is called a Lorentz vector space. A vector $\mathbf{v}$
in a vector space with indefinite scalar product is called%
\[
\begin{array}{lll}
\text{spacelike} & \text{if} & \left\langle \mathbf{x},\mathbf{x}%
\right\rangle _{\mathbb{L}}>0, \\ 
\text{timelike} & \text{if} & \left\langle \mathbf{x},\mathbf{x}%
\right\rangle _{\mathbb{L}}<0 \\ 
\text{lightlike} & \text{if} & \left\langle \mathbf{x},\mathbf{x}%
\right\rangle _{\mathbb{L}}=0,\text{ }\mathbf{x}\neq 0.%
\end{array}%
\]

The Euclidean norm of a vector $\mathbf{x}$ is%
\[
\left\Vert \mathbf{x}\right\Vert =\sqrt{\left\vert \left\langle \mathbf{x},%
\mathbf{x}\right\rangle _{\mathbb{L}}\right\vert }.
\]%
The set of lightlike vectors is called the light cone and zero is spacelike
vector.

Let $\mathcal{V}$ be a finite dimensional real vector space equipped with a
scalar product $\left\langle \cdot ,\cdot \right\rangle $. An orthogonal map
on $\mathcal{V}$ is a linear map $F:\mathcal{V}\longrightarrow \mathcal{V}$
that satisfies%
\begin{equation}
\begin{array}{lll}
\left\langle F\mathbf{u},F\mathbf{v}\right\rangle =\left\langle \mathbf{u},%
\mathbf{v}\right\rangle & \text{for all} & \mathbf{u,v}\in \mathcal{V}\text{.%
}%
\end{array}
\label{orthogonal map}
\end{equation}

Since the scalar product is nondegenerate, this suggests that $F$ is
injective. In this study, since $\dim \mathcal{V}=3$ was assumed to be
finite dimensional, it follows that $F$ is surjective, thus $F$ is
invertible. The orthogonal transformations of a inner product space $(%
\mathcal{V},\left\langle \cdot ,\cdot \right\rangle )$ constitute a group,
this group is denoted by $O(\mathcal{V},\left\langle \cdot ,\cdot
\right\rangle )$. If the determinate of the matrix $A$ corresponding to the
linear map $F$ is $+1$ or $-1$, then this group is called a special
orthogonal group and this group is denoted by $SO(3)$ in Euclidean space and 
$SO(3,1)$ in Lorentz space \cite{ozdemir2005}, \cite{erdogdu2015}, \cite%
{Springborn2021}.

In $\mathbb{R}^{3}$, a quadric can be given by the equatrion%
\begin{equation}
Ax^{2}+By^{2}+Cz^{2}+2Dxy+2Exz+2Fyz+Gx+Dy+Iz+J=0.  \label{quadric}
\end{equation}%
As can be seen, when the first six coefficients are taken as zero, the
equation (\ref{quadric}) indicates a plane. In the most general form, the
equation (\ref{quadric}) denote ellipsoid, hyperboloid, and paraboloid. A
quadratic form a set of points $\left( x,y,z\right) \in \mathbb{R}^{3}$
satisfying%
\[
Q\left( x,y,z\right) =Ax^{2}+By^{2}+Cz^{2}+2Dxy+2Exz+2Fyz
\]%
where $A,B,C,D,E,F\in \mathbb{R}$ with $A,B,C,D,E,F$ not all zero and by $%
\mathbf{M}_{Q}$ its matrix%
\begin{equation}
\mathbf{M}_{Q}=%
\begin{bmatrix}
A & D & E \\ 
D & B & F \\ 
E & F & C%
\end{bmatrix}%
\text{.}  \label{1. katsayilar matrisi}
\end{equation}%
The signature of a matrix $\mathbf{M}_{Q}\in \mathbb{R}^{3\times 3}$ is $%
\left( k,l,m\right) $ where $k$ is the number of strictly positive
eingenvalues of $\mathbf{M}_{Q}$, $l$ is the number of strictly zero
eingenvalues of $\mathbf{M}_{Q}$, $m$\ is the number of strictly negative
eingenvalues of $\mathbf{M}_{Q}$. Namely, the rank of $\mathbf{M}_{Q}$\ is $%
k+m$. For $rank\mathbf{M}_{Q}=3,$ classifies quadratic form in the following
table:%
\begin{equation}
\begin{tabular}{|c|c|}
\hline
Signature of $\mathbf{M}_{Q}$ & Type of quadratic form \\ \hline
$\left( 3,0,0\right) $ or $\left( 0,0,3\right) $ & Ellipsoid \\ \hline
$\left( 2,0,1\right) $ or $\left( 1,0,2\right) $ & cone, 1 or 2-sheeted
hyperboloid \\ \hline
\end{tabular}
\label{ellipsoid}
\end{equation}

The most famous geometric object in which quadratic forms are applied to
geometry is the sphere. If we take the radius of the sphere as the unit of
length, we can do the geometry on the $3$-dimensional unit sphere,%
\begin{equation}
S^{2}=\left\{ \mathbf{x}\in \mathbb{R}^{3}:\left\Vert \mathbf{x}\right\Vert
=1\right\} \subset \mathbb{R}^{3}\text{,}  \label{unit sphere}
\end{equation}%
where $\left\Vert \mathbf{x}\right\Vert =\sqrt{\left\langle \mathbf{x},%
\mathbf{x}\right\rangle _{\mathbb{E}}},$ and $\left\langle \mathbf{x},%
\mathbf{y}\right\rangle _{\mathbb{E}}=\dsum\limits_{i=1}^{3}x_{i}y_{i}$ is
the standard Euclidian scalar product. Another famous sphere is the
3-dimensional Lorentz-Minkowski space sphere%
\begin{equation}
H^{2}=\left\{ \mathbf{x}\in \mathbb{R}^{3}:\left\langle \mathbf{x},\mathbf{x}%
\right\rangle _{\mathbb{L}}=-1\right\} \subset \mathbb{R}^{3}
\label{hiperbolic sphere}
\end{equation}%
where $\left\langle \cdot ,\mathbf{\cdot }\right\rangle _{\mathbb{L}}$
denotes the Lorentz scalar product%
\[
\left\langle \mathbf{x},\mathbf{y}\right\rangle _{\mathbb{L}%
}=-x_{1}y_{1}+x_{2}y_{2}+x_{3}y_{3}.
\]

Generalized scalar product can be found in detail in \"{O}zdemir and \c{S}im%
\c{s}ek articles \cite{ozdemir2016}, \cite{simsek2016}, \cite{simsek2017}.
In addition, the authors also examined linear transformations for
generalized scalar products in the same papers.

\subsection{Quaternions}

Number sets such as quaternions, split quaternions, hybrid numbers is
defined through quadratic forms. These number systems form a geometry
according to the quadratic form in which they are defined. Three-dimensional
Euclidean space and Lorentz-Minkowski space are associated with quaternions
and split quaternions, respectively. The geometric structure of hybrid
numbers is newly studied. These number systems have provided many
conveniences in the geometry to which they are associated. Let's briefly
explain the quaternion, split quaternion and hybrid number systems:

In 1843, Sir William Rowan Hamilton invented the quaternion algebra, which
is familiar denoted $\mathbb{H}$ in his honor. The set of quaternions can be
stated as follows:%
\[
\mathbb{H=}\{a+b\mathbf{i}+c\mathbf{j}+d\mathbf{k}:a,b,c,d\in \mathbb{R}%
\text{, }\mathbf{i}^{2}=\mathbf{j}^{2}=\mathbf{k}^{2}=\mathbf{ijk=}-1\}.
\]%
It was soon realized that the quaternion product could be used to represent
the rotational transformation in $\mathbb{E}^{3}$. In 1855, Arthur Cayley
explored that quaternions could also be used to represent rotations in $%
\mathbb{E}^{4}$. Quaternion algebra has accepted a significant role lately
in different areas of physical science; for example, synthesis of mechanism
and machines, in differential geometry, simulation of particle motion in
molecular physic, in analysis, and quaternionic formulation of the equation
of motion in the theory of relativity. Additionally, quaternions have been
indicated in terms of $4\times 4$ matrices by means of left and right
multiplication operators. If the literature is examined, it can be clearly
seen that the geometric properties of quaternions are a number system
designed to fit the orthonormal frame of $\mathbb{E}^{3}$ \cite{goldman2011}%
, \cite{goldman2010}, \cite{demirci2020}.

The set of split quaternions presented by Cockle \cite{cockle1849} in 1849
is as follows:

\[
\widehat{\mathbb{H}}=\{a+b\mathbf{i}+c\mathbf{j}+d\mathbf{k}:a,b,c,d\in 
\mathbb{R}\text{, }\mathbf{i}^{2}=-1\text{, }\mathbf{j}^{2}=\mathbf{k}^{2}=%
\mathbf{ijk=}1\}.
\]%
Lately, since split quaternions are used to state Lorentzian rotations,
there are studies on geometric and physical applications of split
quaternions that call for solving split quaternionic equations. In general,
split quaternions are used to express rotation and reflection
transformations for timelike and spacelike vectors in Minkowski 3-space.
Because the units of split quaternions are compatible with the unit base
vectors in Minkowski 3-space \cite{erdogdu2015}, \cite{ozdemir2006}, \cite%
{ozdemir2005}, \cite{ozdemir2014}. Rotations around the lightlike axis were
also investigated in \cite{goemans2016}, \cite{nevsovic2016}.

In 2018, \"{O}zdemir invented the set of hybrid numbers \cite{ozdemir2018},
which is a noncommutative ring, is a generalization of complex, hyperbolic
and dual number sets and it is defined as 
\[
\mathbb{K=}\left\{ a+b\mathbf{i}+c\mathbf{\varepsilon }+d\mathbf{h}%
:a,b,c,d\in 
\mathbb{R}
\text{, }\mathbf{i}^{2}=-1\text{, }\mathbf{\varepsilon }^{2}=0\text{, }%
\mathbf{h}^{2}=1\text{, }\mathbf{ih=-hi=\varepsilon }+\mathbf{i}\right\} 
\text{.}
\]%
Generalized quaternions and their algebraic properties are discussed in
detail in \cite{yayli}, \cite{senturk2021}.

\section{Generalized Quadratic Number System}

Quadratic forms are geometric structures whose geometry is studied by
associating them with a number system. Now we will present here a way to
give you all the number systems that can be constructed for all quadratic
forms. Let's define the generalized quadratic number system as follows:

\begin{definition}
The set of generalized quadratic number system, denoted by $\mathfrak{Q}$,
is defined as%
\[
\mathfrak{Q}=\left\{ 
\begin{array}{c}
\mathbf{q}=q_{0}+q_{1}\mathbf{i}+q_{2}\mathbf{j}+q_{3}\mathbf{k:i}^{2}=A,%
\text{ }\mathbf{j}^{2}=B,\text{ }\mathbf{k}^{2}=C, \\ 
\mathbf{ij}=D+\alpha _{1}\mathbf{i}+\alpha _{2}\mathbf{j+}\alpha _{3}\mathbf{%
k,}\text{ }\mathbf{ik}=E+\beta _{1}\mathbf{i}+\beta _{2}\mathbf{j+}\beta _{3}%
\mathbf{k,}\text{ }\mathbf{jk}=F+\lambda _{1}\mathbf{i}+\lambda _{2}\mathbf{%
j+}\lambda _{3}\mathbf{k,} \\ 
\mathbf{ji}=D-\alpha _{1}\mathbf{i}-\alpha _{2}\mathbf{j-}\alpha _{3}\mathbf{%
k,}\text{ }\mathbf{ki}=E-\beta _{1}\mathbf{i}-\beta _{2}\mathbf{j-}\beta _{3}%
\mathbf{k,}\text{ }\mathbf{kj}=F-\lambda _{1}\mathbf{i}-\lambda _{2}\mathbf{%
j-}\lambda _{3}\mathbf{k,}%
\end{array}%
\right\} ,
\]

where $A,B,C,D,E,F\in \mathbb{R}$,%
\[
\text{for }%
\begin{array}{lll}
\alpha _{1}=\lambda _{3}\text{,} & \beta _{1}=-\lambda _{2}\text{,} & \beta
_{3}=-\alpha _{2};%
\end{array}%
\]%
\begin{eqnarray*}
&&%
\begin{array}{lll}
A=\alpha _{2}^{2}+\alpha _{3}\beta _{2}\text{,} & B=\alpha _{1}^{2}-\lambda
_{1}\alpha _{3}\text{,} & C=\beta _{1}^{2}+\lambda _{1}\beta _{2}\text{,}%
\end{array}
\\
&&%
\begin{array}{lll}
D=-\left( \alpha _{1}\alpha _{2}+\alpha _{3}\beta _{1}\right) \text{,} & 
E=-\left( \beta _{2}\alpha _{1}-\alpha _{2}\beta _{1}\right) \text{,} & 
F=\alpha _{1}\beta _{1}+\lambda _{1}\alpha _{2}\text{.}%
\end{array}%
\end{eqnarray*}
\end{definition}

This set of numbers can be thought of as a set of quadruplets defined below:%
\[
\begin{array}{llll}
1\longleftrightarrow \left( 1,0,0,0\right) \mathbf{,} & \mathbf{i}%
\longleftrightarrow \left( 0,1,0,0\right) \mathbf{,} & \mathbf{j}%
\longleftrightarrow \left( 0,0,1,0\right) \mathbf{,} & \mathbf{k}%
\longleftrightarrow \left( 0,0,0,1\right) \text{.}%
\end{array}%
\]%
For a quadratic number $\mathbf{q}=q_{0}+q_{1}\mathbf{i}+q_{2}\mathbf{j}%
+q_{3}\mathbf{k,}$ a real number $q_{0}$ is called the scalar part and is
denoted by $S\left( \mathbf{q}\right) $ and the part $q_{1}\mathbf{i}+q_{2}%
\mathbf{j}+q_{3}\mathbf{k}$ is called the vector part, and is denoted by $%
\mathbf{v}_{\mathbf{q}}$. Now the matrix (\ref{1. katsayilar matrisi}) can
be written as:

\begin{equation}
\mathbf{M}_{\mathfrak{Q}}=%
\begin{bmatrix}
\alpha _{2}^{2}+\alpha _{3}\beta _{2} & -\left( \alpha _{1}\alpha
_{2}+\alpha _{3}\beta _{1}\right) & \alpha _{2}\beta _{1}-\beta _{2}\alpha
_{1} \\ 
-\left( \alpha _{1}\alpha _{2}+\alpha _{3}\beta _{1}\right) & \alpha
_{1}^{2}-\lambda _{1}\alpha _{3} & \lambda _{1}\alpha _{2}+\alpha _{1}\beta
_{1} \\ 
\alpha _{2}\beta _{1}-\beta _{2}\alpha _{1} & \lambda _{1}\alpha _{2}+\alpha
_{1}\beta _{1} & \beta _{1}^{2}+\lambda _{1}\beta _{2}%
\end{bmatrix}%
.  \label{katsayilar matrisi}
\end{equation}

\subsection{Operation in the Quadratic Numbers}

Two quadratic numbers are equal if all their components are equal, one by
one. The sum of two quadratic numbers is defined by summing their
components. Addition operation in the quadratic numbers is both commutative
and associative. Zero is null element. Regarding the addition operation, the
inverse element of $\mathbf{q}$ is $-\mathbf{q}$, which is defined as having
all components of $\mathbf{q}$ changed in their signs. This requires that $%
\left( \mathfrak{Q},+\right) $ is an Abelian group.

The quadratic numbers product%
\[
\mathbf{qp}=\left( q_{0}+q_{1}\mathbf{i}+q_{2}\mathbf{j}+q_{3}\mathbf{k}%
\right) \left( p_{0}+p_{1}\mathbf{i}+p_{2}\mathbf{j}+p_{3}\mathbf{k}\right)
\]%
is obtained by distributing the terms on the right as ordinary algebra.
Given the definition of the set of quadratic numbers, the following
multiplication table can be constructed.

\[
\begin{tabular}{|l|l|l|l|l|}
\hline
$\mathbf{\cdot }$ & $1$ & $\mathbf{i}$ & $\mathbf{j}$ & $\mathbf{k}$ \\ 
\hline
$1$ & $1$ & $\mathbf{i}$ & $\mathbf{j}$ & $\mathbf{k}$ \\ \hline
$\mathbf{i}$ & $\mathbf{i}$ & $A$ & $D+\alpha _{1}\mathbf{i}+\alpha _{2}%
\mathbf{j+}\alpha _{3}\mathbf{k}$ & $E+\beta _{1}\mathbf{i}+\beta _{2}%
\mathbf{j+}\beta _{3}\mathbf{k}$ \\ \hline
$\mathbf{j}$ & $\mathbf{j}$ & $D-\alpha _{1}\mathbf{i}-\alpha _{2}\mathbf{j-}%
\alpha _{3}\mathbf{k}$ & $B$ & $F+\lambda _{1}\mathbf{i}+\lambda _{2}\mathbf{%
j+}\lambda _{3}\mathbf{k}$ \\ \hline
$\mathbf{k}$ & $\mathbf{k}$ & $E-\beta _{1}\mathbf{i}-\beta _{2}\mathbf{j-}%
\beta _{3}\mathbf{k}$ & $F-\lambda _{1}\mathbf{i}-\lambda _{2}\mathbf{j-}%
\lambda _{3}\mathbf{k}$ & $C$ \\ \hline
\end{tabular}%
\]

This table will be used to multiplication any two quadratic numbers. The
table shows us that the multiplication operation in the quadratic numbers is
not commutative. But it has the property of associativity. The conjugate of
a quadratic number $\mathbf{q}=q_{0}+q_{1}\mathbf{i}+q_{2}\mathbf{j}+q_{3}%
\mathbf{k,}$ denoted by $\overline{\mathbf{q}}$, is defined as%
\[
\overline{\mathbf{q}}=q_{0}-q_{1}\mathbf{i}-q_{2}\mathbf{j}-q_{3}\mathbf{k}%
\text{\textbf{.}}
\]%
The conjugate of the sum of quadratic numbers is equal to the sum of their
conjugate:%
\[
\overline{\mathbf{q+p}}=\overline{\mathbf{q}}+\overline{\mathbf{p}}\text{.}
\]%
Besides, according to quadratic numbers product, we have $\mathbf{q}%
\overline{\mathbf{q}}=\overline{\mathbf{q}}\mathbf{q}$ and its value is%
\[
\mathbf{q}\overline{\mathbf{q}}=\overline{\mathbf{q}}\mathbf{q}%
=q^{2}-Ax^{2}-By^{2}-Cz^{2}-2Dxy-2Exz-2Fyz.
\]%
The real number $\sqrt{\left\vert \mathbf{q}\overline{\mathbf{q}}\right\vert 
}$ will be called the norm of the quadratic number $\mathbf{q}$ and will be
denoted by $\left\Vert \mathbf{q}\right\Vert $.

Let's take the quadratic numbers $\mathbf{q}=q_{0}+q_{1}\mathbf{i}+q_{2}%
\mathbf{j}+q_{3}\mathbf{k}$ and $\mathbf{p}=p_{0}+p_{1}\mathbf{i}+p_{2}%
\mathbf{j}+p_{3}\mathbf{k}$. The scalar product of $\mathbf{q}$ and $\mathbf{%
p}$ is defined as%
\begin{eqnarray*}
\left\langle \mathbf{q},\mathbf{p}\right\rangle _{\mathfrak{Q}} &=&\frac{%
\mathbf{q}\overline{\mathbf{p}}+\mathbf{p}\overline{\mathbf{q}}}{2} \\
&=&p_{0}q_{0}-Ap_{1}q_{1}-Bp_{2}q_{2}-Cp_{3}q_{3}-Fp_{2}q_{3}-Fp_{3}q_{2}-Dp_{1}q_{2}-Dp_{2}q_{1}-Ep_{1}q_{3}-Ep_{3}q_{1}
\\
&=&p_{0}q_{0}-\left( Ap_{1}q_{1}+Bp_{2}q_{2}+Cp_{3}q_{3}+D\left(
p_{1}q_{2}+p_{2}q_{1}\right) +E\left( p_{1}q_{3}+p_{3}q_{1}\right) +F\left(
p_{2}q_{3}+p_{3}q_{2}\right) \right) \text{.}
\end{eqnarray*}%
So, scalar product of quadratic vectors $\mathbf{v}=\left(
v_{1},v_{2},v_{3}\right) $ and $\mathbf{u}=\left( u_{1},u_{2},u_{3}\right) $
is%
\[
\left\langle \mathbf{u},\mathbf{v}\right\rangle _{\mathfrak{Q}}=\Delta
\left( Au_{1}v_{1}+Bu_{2}v_{2}+Cu_{3}v_{3}+D\left(
u_{1}v_{2}+u_{2}v_{1}\right) +E\left( u_{1}v_{3}+u_{3}v_{1}\right) +F\left(
u_{2}v_{3}+u_{3}v_{2}\right) \right)
\]%
where $\Delta =-1$ if the quadratic form corresponding to $\left\langle 
\mathbf{u},\mathbf{v}\right\rangle _{\mathfrak{Q}}$ is an ellipsoid, and $%
\Delta =1$ if the quadratic form corresponding to $\left\langle \mathbf{u},%
\mathbf{v}\right\rangle _{\mathfrak{Q}}$ is hyperboloid. Thus, the matrix
associated with the bilinear form $\left\langle \mathbf{u},\mathbf{v}%
\right\rangle _{\mathfrak{Q}}$ becomes%
\begin{equation}
\mathfrak{M}_{\mathfrak{Q}}=\Delta M_{\mathfrak{Q}}=\Delta 
\begin{bmatrix}
A & D & E \\ 
D & B & F \\ 
E & F & C%
\end{bmatrix}%
.  \label{deltali matris}
\end{equation}%
While obtaining this matrix (\ref{deltali matris}), the type of quadratic
form can be determined by looking at the sign of the eigenvalues of the
matrix (\ref{1. katsayilar matrisi}). After that, we'll work with vectors in
the quadratic 3-space and find the norm of the vectors with the help of this
scalar product. Also, the vector product of two quadratic numbers $\mathbf{q}
$ and $\mathbf{p}$ defined as%
\begin{eqnarray*}
\mathbf{q}\times _{\mathfrak{Q}}\mathbf{p} &=&\frac{\mathbf{q}\overline{%
\mathbf{p}}-\mathbf{p}\overline{\mathbf{q}}}{2} \\
&=&\left( p_{0}q_{1}-p_{1}q_{0}+\alpha _{1}\left(
p_{1}q_{2}-p_{2}q_{1}\right) +\beta _{1}\left( p_{1}q_{3}-p_{3}q_{1}\right)
+\lambda _{1}\left( p_{2}q_{3}-p_{3}q_{2}\right) \right) \mathbf{i+} \\
&&+\left( p_{0}q_{2}-p_{2}q_{0}+\alpha _{2}\left(
p_{1}q_{2}-p_{2}q_{1}\right) +\beta _{1}\left( p_{3}q_{2}-p_{2}q_{3}\right)
+\beta _{2}\left( p_{1}q_{3}-p_{3}q_{1}\right) \right) \mathbf{j+} \\
&&+\left( p_{0}q_{3}-q_{0}p_{3}+\alpha _{1}\left(
p_{2}q_{3}-p_{3}q_{2}\right) +\alpha _{2}\left( p_{3}q_{1}-p_{1}q_{3}\right)
+\alpha _{3}\left( p_{1}q_{2}-p_{2}q_{1}\right) \right) \mathbf{k}\text{.}
\end{eqnarray*}

\begin{definition}
For $\mathbf{u}=\left( u_{1},u_{2},u_{3}\right) ,$ $\mathbf{v}=\left(
v_{1},v_{2},v_{3}\right) ,$ the vector product in quadratic 3-space is%
\[
\mathbf{u}\times \mathbf{v}=\left\vert 
\begin{array}{ccc}
\lambda _{1}\mathbf{i}-\beta _{1}\mathbf{j+}\alpha _{1}\mathbf{k} & -\beta
_{1}\mathbf{i}-\beta _{2}\mathbf{j}+\alpha _{2}\mathbf{k} & \alpha _{1}%
\mathbf{i}+\alpha _{2}\mathbf{j+}\alpha _{3}\mathbf{k} \\ 
u_{1} & u_{2} & u_{3} \\ 
v_{1} & v_{2} & v_{3}%
\end{array}%
\right\vert
\]%
Therefore, the quadratic product of $\mathbf{q}=S_{\mathbf{q}}+\mathbf{v}_{%
\mathbf{q}}$ and $\mathbf{p}=S_{\mathbf{p}}+\mathbf{v}_{\mathbf{p}}\ $can be
written as%
\begin{equation}
\mathbf{qp}=S_{\mathbf{q}}S_{\mathbf{p}}+S_{\mathbf{q}}\mathbf{v}_{\mathbf{p}%
}+S_{\mathbf{p}}\mathbf{v}_{\mathbf{q}}+\Delta \left\langle \mathbf{v}_{%
\mathbf{q}},\mathbf{v}_{\mathbf{p}}\right\rangle +\mathbf{v}_{\mathbf{q}%
}\times \mathbf{v}_{\mathbf{p}}.  \label{quadraticcarpim}
\end{equation}%
Actually,%
\begin{eqnarray*}
\mathbf{qp} &=&%
\begin{bmatrix}
q_{0} & Aq_{1}+Dq_{2}+Eq_{3} & Bq_{2}+Dq_{1}+Fq_{3} & Cq_{3}+Eq_{1}+Fq_{2}
\\ 
q_{1} & q_{0}-q_{2}\alpha _{1}-q_{3}\beta _{1} & q_{1}\alpha
_{1}-q_{3}\lambda _{1} & \lambda _{1}q_{2}+q_{1}\beta _{1} \\ 
q_{2} & -q_{2}\alpha _{2}-\beta _{2}q_{3} & q_{0}+q_{1}\alpha
_{2}+q_{3}\beta _{1} & \beta _{2}q_{1}-\beta _{1}q_{2} \\ 
q_{3} & -\alpha _{3}q_{2}+q_{3}\alpha _{2} & \alpha _{3}q_{1}-q_{3}\alpha
_{1} & q_{0}-q_{1}\alpha _{2}+q_{2}\alpha _{1}%
\end{bmatrix}%
\begin{bmatrix}
p_{0} \\ 
p_{1} \\ 
p_{2} \\ 
p_{3}%
\end{bmatrix}
\\
&=&%
\begin{bmatrix}
p_{0}q_{0}+Ap_{1}q_{1}+Bp_{2}q_{2}+Cp_{3}q_{3}+D\left(
p_{1}q_{2}+p_{2}q_{1}\right) +E\left( p_{1}q_{3}+p_{3}q_{1}\right) +F\left(
p_{2}q_{3}+p_{3}q_{2}\right) \\ 
p_{0}q_{1}+p_{1}q_{0}-\alpha _{1}p_{1}q_{2}+\alpha _{1}p_{2}q_{1}-\beta
_{1}p_{1}q_{3}+\beta _{1}p_{3}q_{1}-\lambda _{1}p_{2}q_{3}+\lambda
_{1}p_{3}q_{2} \\ 
p_{0}q_{2}+p_{2}q_{0}-\alpha _{2}p_{1}q_{2}+\alpha _{2}p_{2}q_{1}+\beta
_{1}p_{2}q_{3}-\beta _{1}p_{3}q_{2}-\beta _{2}p_{1}q_{3}+\beta _{2}p_{3}q_{1}
\\ 
p_{0}q_{3}+q_{0}p_{3}-\alpha _{1}p_{2}q_{3}+\alpha _{1}p_{3}q_{2}+\alpha
_{2}p_{1}q_{3}-\alpha _{2}p_{3}q_{1}-\alpha _{3}p_{1}q_{2}+\alpha
_{3}p_{2}q_{1}%
\end{bmatrix}%
.
\end{eqnarray*}%
and also,%
\[
\left( \mathbf{u}\times \mathbf{v}\right) \times \mathbf{w}=\Delta
\left\langle \mathbf{v,w}\right\rangle \mathbf{u}-\Delta \left\langle 
\mathbf{\mathbf{u},w}\right\rangle \mathbf{v.}
\]
\end{definition}

This scalar and vector products on $\mathbb{R}^{3}$ allows us to define a
new metric space. We will call it the quadratic metric.

\section{Geometry of Non-Degenerate Quadratic Metrics}

This section will focus on two types of quadratic spaces derived from
quadratic metrics. These are the ellipsoid quadratic space and the
hyperboloid quadratic space.

\section{Ellipsoid Quadratic Spaces}

\begin{definition}
For a quadratic number $\mathbf{X}=q+x\mathbf{i}+y\mathbf{j}+z\mathbf{k}$,
the vector,%
\[
\mathbf{v}_{\mathbf{X}}=\left( x,y,z\right) \in \mathbb{R}^{3}
\]%
is called the quadratic vector of $\mathbf{X}$ and a pure quadratic number $x%
\mathbf{i}+y\mathbf{j}+z\mathbf{k}$ will be taken as a vector of quadratic
3-space.
\end{definition}

\begin{definition}
The unit ellipsoid sphere is the set of all ellipsoid vector of $\mathbb{R}%
^{3}$ is%
\[
\mathfrak{E}=\left\{ \left( x,y,z\right) \in \mathbb{R}%
^{3}:Ax^{2}+By^{2}+Cz^{2}+2Dxy+2Exz+2Fyz=-1\right\} \cup \left\{ \left(
0,0,0\right) \right\} .
\]
\end{definition}

\begin{definition}
If the signature of the matrix (\ref{deltali matris}) associated with the
bilinear form $\left\langle \mathbf{\cdot },\mathbf{\cdot }\right\rangle _{%
\mathfrak{Q}}$ is $(3,0,0)$, then $\left\langle \mathbf{\cdot },\mathbf{%
\cdot }\right\rangle _{\mathfrak{Q}}$ is a positive definite bilinear form.
Therefore, the quadratic form corresponding to the bilinear form is
ellipsoid. So, the quadratic number $\mathbf{X}=q+x\mathbf{i}+y\mathbf{j}+z%
\mathbf{k}$ is called the ellipsoid number, and $\mathbf{v}_{\mathbf{X}%
}=\left( x,y,z\right) $ is called the ellipsoid vector. We will denote the
set of ellipsoid quadratic number by $\mathfrak{EQ.}$
\end{definition}

A scalar product on a real vector space $\mathcal{V}$ is a non-degenerate
symmetric bilinear form. A positively definite scalar product is a Euclidean
scalar product and its sign is $(3,0,0)$. The space equipped with a
Euclidean scaler product is called a Euclidean vector space \cite%
{Springborn2021}.

An Euclidean scalar product is a scalar product that is positive definite,
and its one with signature $(3,0,0)$. A vector space equipped with a
Euclidean scalar product is called a Euclidean vector space.

\begin{definition}
The ellipsoid quadratic space is the vector space $\mathbb{R}^{3}$ equipped
with the positive definite symmetric bilinear form%
\begin{eqnarray*}
\left\langle \mathbf{u},\mathbf{v}\right\rangle _{\mathfrak{EQ}} &=&\Delta
\left( Au_{1}v_{1}+Bu_{2}v_{2}+Cu_{3}v_{3}+D\left(
u_{1}v_{2}+u_{2}v_{1}\right) +E\left( u_{1}v_{3}+u_{3}v_{1}\right) +F\left(
u_{2}v_{3}+u_{3}v_{2}\right) \right) \\
&=&%
\begin{bmatrix}
u_{1} \\ 
u_{2} \\ 
u_{3}%
\end{bmatrix}%
^{T}%
\begin{bmatrix}
-A & -D & -E \\ 
-D & -B & -F \\ 
-E & -F & -C%
\end{bmatrix}%
\begin{bmatrix}
v_{1} \\ 
v_{2} \\ 
v_{3}%
\end{bmatrix}%
\end{eqnarray*}%
for the vectors $\mathbf{u}=\left( u_{1},u_{2},u_{3}\right) $, $\mathbf{v}%
=\left( v_{1},v_{2},v_{3}\right) $. The matrix associated with the symmetric
bilinear form $\left\langle \cdot ,\cdot \right\rangle _{\mathfrak{EQ}}$ is $%
-\mathbf{M}_{\mathfrak{Q}}$. Throughout the article, a pure quadratic number
will be taken as a vector in the quadratic 3-space. If $\left\langle \mathbf{%
u},\mathbf{v}\right\rangle _{\mathfrak{EQ}}=0,$ we call them $\mathbf{u}$
and $\mathbf{v}$ are orthogonal to each other. Also, norm of the vector $%
\mathbf{u}$ is%
\[
\left\Vert \mathbf{u}\right\Vert _{\mathfrak{EQ}}=\sqrt{\left\langle \mathbf{%
u},\mathbf{u}\right\rangle _{\mathfrak{EQ}}}.
\]%
It is also very obvious that the ellipsoid quadratic number forms a group by
multiplication. We can denote this group with%
\[
\mathfrak{EQ}=\left\{ \mathbf{q}\in \mathfrak{EQ}:\mathcal{C}\left( \mathbf{q%
}\right) >0\right\} \text{.}
\]
\end{definition}

\begin{proposition}
The polar form of the ellipsoid quadratic number $\mathbf{q}=q+x\mathbf{i}+y%
\mathbf{j}+z\mathbf{k}$ is%
\[
\mathbf{q=\left\Vert \mathbf{q}\right\Vert }\left( \cos \theta +\mathbf{v}%
\sin \theta \right)
\]%
where $\mathbf{v}$ is an ellipsoid vector, $\mathbf{v}=\frac{\mathbf{v}_{%
\mathbf{q}}}{\left\Vert \mathbf{v}_{\mathbf{q}}\right\Vert }$, $\mathbf{v}%
^{2}=-1$, and $\theta $ is the argument of $\mathbf{q}$ defined as%
\[
\theta =\left\{ 
\begin{array}{ll}
\pi -\arctan \frac{\left\Vert \mathbf{v}_{\mathbf{q}}\right\Vert }{%
\left\vert q\right\vert } & ,\text{ }q<0; \\ 
\arctan \frac{\left\Vert \mathbf{v}_{\mathbf{q}}\right\Vert }{\left\vert
q\right\vert } & ,\text{ }q>0;%
\end{array}%
\right. \text{.}
\]
\end{proposition}

\section{Indefinite Quadratic Spaces}

\begin{definition}
The $1-$sheeted hyperboloid is the set of all spacelike vector of%
\[
\mathcal{H}^{1}=\left\{ \left( x,y,z\right) \in \mathbb{R}%
^{3}:Ax^{2}+By^{2}+Cz^{2}+2Dxy+2Exz+2Fyz=1\right\} \cup \left\{ \left(
0,0,0\right) \right\} .
\]%
The $2-$sheeted hyperboloid is the set of all timelike vector of%
\[
\mathcal{H}^{2}=\left\{ \left( x,y,z\right) \in \mathbb{R}%
^{3}:Ax^{2}+By^{2}+Cz^{2}+2Dxy+2Exz+2Fyz=-1\right\} .
\]%
The cone is the set of all lightlike vector of%
\[
\mathcal{C}=\left\{ \left( x,y,z\right) \in \mathbb{R}%
^{3}:Ax^{2}+By^{2}+Cz^{2}+2Dxy+2Exz+2Fyz=0\right\} -\left\{ \left(
0,0,0\right) \right\} .
\]
\end{definition}

The real numbers%
\[
\mathcal{C}\left( \mathbf{q}\right) =\mathbf{q}\overline{\mathbf{q}}=%
\overline{\mathbf{q}}\mathbf{q}=q^{2}-Ax^{2}-By^{2}-Cz^{2}-2Dxy-2Exz-2Fyz
\]%
and%
\[
\mathfrak{V}(\mathbf{q})=Ax^{2}+By^{2}+Cz^{2}+2Dxy+2Exz+2Fyz>0
\]%
are called the character of the indefinite number $\mathbf{q}=q+x\mathbf{i}+y%
\mathbf{j}+z\mathbf{k}$ and the type of the indefinite vector $x\mathbf{i}+y%
\mathbf{j}+z\mathbf{k}$, respectively. Accordingly, we can give the
following definitions.

\begin{definition}
If the matrix (\ref{deltali matris}) associated with the bilinear form $%
\left\langle \mathbf{u},\mathbf{v}\right\rangle _{\mathfrak{Q}}$ is
indefinite form, then the quadratic number $\mathbf{X}$ is called indefinite
number. So, the quadratic number $\mathbf{X}=q+x\mathbf{i}+y\mathbf{j}+z%
\mathbf{k}$ is called the indefinite number, and $\mathbf{v}_{\mathbf{X}%
}=\left( x,y,z\right) $ is called the indefinite vector. We will denote the
set of indefinite quadratic number by $\mathfrak{IQ.}$A indefinite number is
called spacelike, lightlike or timelike, depending on whether $\mathcal{C}%
\left( \mathbf{X}\right) <0$, $\mathcal{C}\left( \mathbf{X}\right) =0$ or $%
\mathcal{C}\left( \mathbf{X}\right) >0$, respectively. The indefinite vector
is called timelike (2-sheeted hyperbolic vector), lightlike (cone vector) or
spacelike (1-sheeted hyperbolic vector), if $\mathfrak{V}_{\mathbf{v}}\left( 
\mathbf{X}\right) <0,$ $\mathfrak{V}_{\mathbf{v}}\left( \mathbf{X}\right)
=0, $ $\mathfrak{V}_{\mathbf{v}}\left( \mathbf{X}\right) >0$, respectively.
We will denote the set of indefinite quadratic number by $\mathfrak{IQ.}$
\end{definition}

A indefinite scaler product with the signature $(2,0,1)$, with $1$ negative
index is a Lorentz scaler product. A vector space equipped with a Lorentz
scalar product is called a Lorentz vector space \cite{Springborn2021}.

\begin{definition}
The quadratic $3$-space is the vector space $\mathbb{R}^{3}$ equipped with
the indefinite symmetric bilinear form%
\begin{eqnarray*}
\left\langle \mathbf{u},\mathbf{v}\right\rangle _{\mathfrak{IQ}}
&=&Au_{1}v_{1}+Bu_{2}v_{2}+Cu_{3}v_{3}+D\left( u_{1}v_{2}+u_{2}v_{1}\right)
+E\left( u_{1}v_{3}+u_{3}v_{1}\right) +F\left( u_{2}v_{3}+u_{3}v_{2}\right)
\\
&=&%
\begin{bmatrix}
u_{1} \\ 
u_{2} \\ 
u_{3}%
\end{bmatrix}%
^{T}%
\begin{bmatrix}
A & D & E \\ 
D & B & F \\ 
E & F & C%
\end{bmatrix}%
\begin{bmatrix}
v_{1} \\ 
v_{2} \\ 
v_{3}%
\end{bmatrix}%
\end{eqnarray*}%
for the vectors $\mathbf{u}=\left( u_{1},u_{2},u_{3}\right) $, $\mathbf{v}%
=\left( v_{1},v_{2},v_{3}\right) $. The matrix associated with the symmetric
bilinear form $\left\langle \cdot ,\cdot \right\rangle _{\mathfrak{IQ}}$ is $%
\mathbf{M}_{\mathfrak{Q}}$. If $\left\langle \mathbf{u},\mathbf{v}%
\right\rangle _{\mathfrak{IQ}}=0,$ we call them $\mathbf{u}$ and $\mathbf{v}$
are pseudo-orthogonal. Also, norm of the vector $\mathbf{u}$ is%
\[
\left\Vert \mathbf{u}\right\Vert _{\mathfrak{IQ}}=\sqrt{\left\vert
\left\langle \mathbf{u},\mathbf{u}\right\rangle _{\mathfrak{IQ}}\right\vert }%
.
\]
\end{definition}

The indefinite vector of a indefinite quadratic number can be timelike,
spacelike or lightlike. Also, the indefinite vector of a spacelike quadratic
number will definitely be spacelike. We can easily see that%
\[
\mathcal{C}_{\mathbf{v}}(\mathbf{q})=Ax^{2}+By^{2}+Cz^{2}+2Dxy+2Exz+2Fyz>0
\]%
from the inequality%
\begin{eqnarray*}
\mathcal{C}\left( \mathbf{q}\right) &<&0\Longrightarrow
q^{2}-Ax^{2}-By^{2}-Cz^{2}-2xyD-2xzE-2Fyz<0 \\
&\Longrightarrow &Ax^{2}+By^{2}+Cz^{2}+2Dxy+2Exz+2Fyz>q^{2}>0\text{.}
\end{eqnarray*}%
It means that the indefinite vector type of a spacelike quadratic number is
certainly spacelike. Similarly, the indefinite vector of a lightlike
quadratic number is definitely a spacelike vector if the scalar part of
lightlike quadratic number is non-zero, and lightlike vector if scalar part
is zero. So, the type of a lightlike indefinite quadratic number is either
spacelike or lightlike. Therefore, we can give the following table.%
\[
\begin{tabular}{|l|l|l|l|}
\hline
$\mathbf{q}$ & \textbf{Spacelike} & \textbf{Lightlike} & \textbf{Timelike}
\\ \hline
& Spacelike & Spacelike & Spacelike \\ \cline{2-4}
$\mathbf{v}_{\mathbf{q}}$ &  & Lightlike & Lightlike \\ \cline{2-4}
&  &  & Timelike \\ \hline
\end{tabular}%
\]

Also, using the product of the quadratic number, we can obtain the following
equation:%
\[
\mathcal{C}\left( \mathbf{pq}\right) =\mathcal{C}\left( \mathbf{p}\right) 
\mathcal{C}\left( \mathbf{q}\right) .
\]%
Therefore, the timelike indefinite quadratic number forms a group by
multiplication. We can denote this group with%
\[
\mathfrak{TQ}=\left\{ \mathbf{q}\in \mathfrak{HQ}:\mathcal{C}\left( \mathbf{q%
}\right) >0\right\} \text{.}
\]

The following table can be created according to the timelike indefinite
quadratic number multiplication operation.

\[
\begin{tabular}{|l|l|l|l|}
\hline
$\mathbf{\cdot }$ & \textbf{Spacelike} & \textbf{Timelike} & \textbf{%
Lightlike} \\ \hline
\textbf{Spacelike} & Timelike & Spacelike & Lightlike \\ \hline
\textbf{Timelike} & Spacelike & Timelike & Lightlike \\ \hline
\textbf{Lightlike} & Lightlike & Lightlike & Lightlike \\ \hline
\end{tabular}%
\]

\begin{proposition}
The polar form of the spacelike quadratic number $\mathbf{q}=q+x\mathbf{i}+y%
\mathbf{j}+z\mathbf{k}$ is%
\[
\mathbf{q}=\mathbf{\left\Vert \mathbf{q}\right\Vert }\left( \sinh \theta +%
\mathbf{v}\cosh \theta \right)
\]%
where $\mathbf{v}$ is a spacelike vector, $\mathbf{v}=\dfrac{\mathbf{v}_{%
\mathbf{q}}}{\left\Vert \mathbf{v}_{\mathbf{q}}\right\Vert }$, $\mathbf{v}%
^{2}=1$.
\end{proposition}

\begin{proposition}
The polar form of the timelike quadratic number $\mathbf{q}=q+x\mathbf{i}+y%
\mathbf{j}+z\mathbf{k}$ is%
\[
\begin{array}{ll}
\mathbf{q}=\left\Vert \mathbf{q}\right\Vert \left( \epsilon \cosh \theta +%
\mathbf{v}\sinh \theta \right) \text{,} & \text{if }\mathbf{v}\text{ is
spacelike} \\ 
\mathbf{q}=\left\Vert \mathbf{q}\right\Vert \left( \cos \theta +\mathbf{v}%
\sin \theta \right) \text{,} & \text{if }\mathbf{v}\text{ is timelike,} \\ 
\mathbf{q}=\left\vert q\right\vert \left( \epsilon +\mathbf{v}\right) \text{,%
} & \text{if }\mathbf{v}\text{ is lightlike,}%
\end{array}%
\]%
where $\mathbf{v}=\dfrac{\mathbf{v}_{\mathbf{q}}}{\mathbf{\left\Vert \mathbf{%
v}_{\mathbf{q}}\right\Vert }}$ and $\epsilon =$sign$q$.
\end{proposition}

\begin{proposition}
The polar form of the lightlike quadratic number $\mathbf{q}=q+x\mathbf{i}+y%
\mathbf{j}+z\mathbf{k}$ is%
\[
\begin{array}{cc}
\mathbf{q}=\left\vert q\right\vert \left( \epsilon +\mathbf{v}\right) \text{,%
} & \text{if }\mathbf{v}\text{ is spacelike,}%
\end{array}%
\]%
where $\mathbf{v}=\dfrac{\mathbf{v}_{\mathbf{q}}}{\left\vert q\right\vert }$%
, $\epsilon =$sign$q$ and%
\[
\begin{array}{cc}
\mathbf{q}=\mathbf{v}_{\mathbf{q}}\text{,} & \text{if }\mathbf{v}\text{ is
lightlike.}%
\end{array}%
\]
\end{proposition}

\begin{proposition}
The argument of a non-lightlike quadratic number $\mathbf{q}=q+x\mathbf{i}+y%
\mathbf{j}+z\mathbf{k}$ is defined as%
\[
\theta =\left\{ 
\begin{array}{cc}
\pi -\arctan \dfrac{\mathbf{\left\Vert \mathbf{v}_{\mathbf{q}}\right\Vert }}{%
\left\vert q\right\vert }\text{,} & q_{1}<0 \\ 
\arctan \dfrac{\mathbf{\left\Vert \mathbf{v}_{\mathbf{q}}\right\Vert }}{%
\left\vert q\right\vert }\text{,} & q_{1}>0%
\end{array}%
\right.
\]%
$\allowbreak $and%
\[
\begin{array}{cc}
\theta =\ln \left( \dfrac{\left\vert q\right\vert +\mathbf{\left\Vert 
\mathbf{v}_{\mathbf{q}}\right\Vert }}{\mathbf{\left\Vert q\right\Vert }}%
\right) \text{,} & \text{and\ }\theta =1%
\end{array}%
\]%
for timelike, spacelike and lightlike quadratic numbers, respectively.
\end{proposition}

\section{Rotation with Quadratic Numbers}

Before moving on to the geometric properties of quadratic form
multiplication, let's introduce the left multiplication and right
multiplication matrices. Let the left multiplication and right
multiplication matrices be $L$ and $R$, respectively. For $\mathbf{q}=q+x%
\mathbf{i+}y\mathbf{j+}z\mathbf{k,}$ these matrices are as follows:%
\begin{eqnarray}
L\left( \mathbf{q}\right) &=&%
\begin{bmatrix}
q & yD+zE+Ax & xD+By+Fz & xE+Cz+Fy \\ 
x & q-y\alpha _{1}-z\beta _{1} & x\alpha _{1}-z\lambda _{1} & x\beta
_{1}+y\lambda _{1} \\ 
y & -y\alpha _{2}-z\beta _{2} & q+x\alpha _{2}+z\beta _{1} & x\beta
_{2}-y\beta _{1} \\ 
z & z\alpha _{2}-y\alpha _{3} & x\alpha _{3}-z\alpha _{1} & q-x\alpha
_{2}+y\alpha _{1}%
\end{bmatrix}%
,  \label{left matrix} \\
R\left( \mathbf{q}\right) &=&%
\begin{bmatrix}
q & yD+zE+Ax & xD+By+Fz & xE+Cz+Fy \\ 
x & q+y\alpha _{1}+z\beta _{1} & z\lambda _{1}-x\alpha _{1} & -x\beta
_{1}-y\lambda _{1} \\ 
y & y\alpha _{2}+z\beta _{2} & q-x\alpha _{2}-z\beta _{1} & y\beta
_{1}-x\beta _{2} \\ 
z & y\alpha _{3}-z\alpha _{2} & z\alpha _{1}-x\alpha _{3} & q+x\alpha
_{2}-y\alpha _{1}%
\end{bmatrix}%
\text{.}  \label{right matrix}
\end{eqnarray}%
and their eigenvectors are%
\begin{eqnarray*}
&&q+\sqrt{Ax^{2}+By^{2}+Cz^{2}+2Dxy+2Exz+2Fyz}, \\
&&q-\sqrt{Ax^{2}+By^{2}+Cz^{2}+2Dxy+2Exz+2Fyz}.
\end{eqnarray*}%
At the same time, the skew symmetric (or semi skew symmetric) matrix
corresponding to the unit pure quadratic number $\mathbf{v}_{\mathbf{q}%
}=\left( x,y,z\right) $ is%
\[
\mathfrak{S}=%
\begin{bmatrix}
-\alpha _{1}y-\beta _{1}z & \alpha _{1}x-\lambda _{1}z & \beta _{1}x+\lambda
_{1}y \\ 
-\beta _{2}z-\alpha _{2}y & \alpha _{2}x+\beta _{1}z & \beta _{2}x-\beta
_{1}y \\ 
\alpha _{2}z-\alpha _{3}y & \alpha _{3}x-\alpha _{1}z & \alpha _{1}y-\alpha
_{2}x%
\end{bmatrix}%
.
\]%
This matrix will be used to obtain rotation matrices.

\begin{theorem}
$S_{\mathbf{q}}(\mathbf{p})=\mathbf{qp}\overline{\mathbf{q}}$ is a rotation
transformation where $\mathbf{q}=q+x\mathbf{i+}y\mathbf{j+}z\mathbf{k}$ is
the unit ellipsoid (or\ unit timelike indefinite quadratic number) and $%
\mathbf{p}$ is ellipsoid (or\ unit timelike indefinite quadratic number).
\end{theorem}

\begin{proof}
Using (\ref{left matrix}) and (\ref{right matrix}), the following equation
can be written:%
\[
S_{\mathbf{q}}(\mathbf{p})=\mathbf{qp}\overline{\mathbf{q}}=L\left( \mathbf{q%
}\right) R\left( \overline{\mathbf{q}}\right) \mathbf{p.}
\]%
We have the matrix \textrm{R}$_{\mathbf{v}}^{\theta }=L\left( \mathbf{q}%
\right) R\left( \overline{\mathbf{q}}\right) $ that is%
\begin{equation}
\mathrm{R}_{\mathbf{v}}^{\theta }=%
\begin{bmatrix}
1 & 0 \\ 
0 & \mathrm{M}_{\mathbf{v}}^{\theta }%
\end{bmatrix}
\label{donme matrisi1}
\end{equation}%
$\allowbreak $where%
\[
\mathrm{M}_{\mathbf{v}}^{\theta }=%
\begin{bmatrix}
\begin{array}{c}
-Aq_{1}^{2}+Bq_{2}^{2}+Cq_{3}^{2} \\ 
+2Fq_{2}q_{3}+q_{0}^{2}+ \\ 
-2\alpha _{1}q_{0}q_{2}-2\beta _{1}q_{0}q_{3}%
\end{array}
& 
\begin{array}{c}
-2Dq_{1}^{2}-2Bq_{1}q_{2} \\ 
-2Fq_{1}q_{3}+2\alpha _{1}q_{0}q_{1} \\ 
-2\lambda _{1}q_{0}q_{3}%
\end{array}
& 
\begin{array}{c}
-2Eq_{1}^{2}-2Cq_{1}q_{3} \\ 
-2Fq_{1}q_{2}+2\beta _{1}q_{0}q_{1} \\ 
+2\lambda _{1}q_{0}q_{2}%
\end{array}
\\ 
\begin{array}{c}
-2Dq_{2}^{2}-2Aq_{1}q_{2} \\ 
-2Eq_{2}q_{3}-2\alpha _{2}q_{0}q_{2} \\ 
-2\beta _{2}q_{0}q_{3}%
\end{array}
& 
\begin{array}{c}
Aq_{1}^{2}-Bq_{2}^{2}+Cq_{3}^{2} \\ 
+2Eq_{1}q_{3}+q_{0}^{2} \\ 
+2\alpha _{2}q_{0}q_{1}+2\beta _{1}q_{0}q_{3}%
\end{array}
& 
\begin{array}{c}
-2Fq_{2}^{2}-2Cq_{2}q_{3} \\ 
-2Eq_{1}q_{2}-2\beta _{1}q_{0}q_{2} \\ 
+2\beta _{2}q_{0}q_{1}%
\end{array}
\\ 
\begin{array}{c}
-2Eq_{3}^{2}-2Aq_{1}q_{3} \\ 
-2Dq_{2}q_{3}+2\alpha _{2}q_{0}q_{3} \\ 
-2\alpha _{3}q_{0}q_{2}%
\end{array}
& 
\begin{array}{c}
-2Fq_{3}^{2}-2Dq_{1}q_{3} \\ 
-2Bq_{2}q_{3}+2\alpha _{3}q_{0}q_{1} \\ 
-2\alpha _{1}q_{0}q_{3}%
\end{array}
& 
\begin{array}{c}
q_{0}^{2}Aq_{1}^{2}+Bq_{2}^{2}-Cq_{3}^{2}+ \\ 
+2Dq_{1}q_{2}+2\alpha _{1}q_{0}q_{2} \\ 
-2\alpha _{2}q_{0}q_{1}%
\end{array}%
\end{bmatrix}%
.
\]%
For any unit ellipsoid number or unit non lightlike number $\mathbf{q}$, one
of the eigenvectors of the matrix (\ref{donme matrisi1}) is $x\mathbf{i+}y%
\mathbf{j+}z\mathbf{k}$ and the corresponding eigenvalue is $1$. At the same
time, $\det $\textrm{R}$=1$.
\end{proof}

\subsection{Rodrigues Rotation Formula}

The exponential map is defined by the matrix exponential series $e^{%
\mathfrak{S}}$. For any skew-symmetric\textbf{\ }matrix $\mathfrak{S}$, the
matrix exponential $e^{\mathfrak{S}}$ always gives a rotation matrix. This
method is known as the Rodrigues formula. The Rodrigues rotation formula is
a benefit procedure for generating a rotation matrix \cite{erdogdu2015}, 
\cite{nevsovic2016}, \cite{ozdemir2016}. We can use it to obtain a rotation
matrix in the quadratic 3-space. For this, the skew-symmetric matrix with
rotation axis $\mathbf{v}=x\mathbf{i+}y\mathbf{j+}z\mathbf{k}$ in the
quadratic $3$-space is%
\begin{equation}
\mathfrak{S}=%
\begin{bmatrix}
-\alpha _{1}y-\beta _{1}z & \alpha _{1}x-\lambda _{1}z & \beta _{1}x+\lambda
_{1}y \\ 
-\beta _{2}z-\alpha _{2}y & \alpha _{2}x+\beta _{1}z & \beta _{2}x-\beta
_{1}y \\ 
\alpha _{2}z-\alpha _{3}y & \alpha _{3}x-\alpha _{1}z & \alpha _{1}y-\alpha
_{2}x%
\end{bmatrix}%
\text{.}  \label{skewsymquadric}
\end{equation}%
$\allowbreak $ $\allowbreak $

\begin{theorem}
Let $\mathfrak{S}$ be a skew symmetric matrix in the form (\ref%
{skewsymquadric}) and $\mathbf{v}=x\mathbf{i+}y\mathbf{j+}z\mathbf{k}$ is a
quadratic vector and on the quadric surface $\mathfrak{Q}$. Then the matrix
exponential%
\[
\begin{tabular}{ll}
$\mathfrak{R}_{\mathbf{v}}^{\theta }=e^{\theta \mathfrak{S}}=I_{3}+\left(
\sin \theta \right) \mathfrak{S}+\left( 1-\cos \theta \right) \mathfrak{S}%
^{2},$ $\theta \in \lbrack 0,2\pi )$ & ,$\text{if }\mathbf{v}\text{ is a
unit ellipsoid or timelike}$ \\ 
$\mathfrak{R}_{\mathbf{v}}^{\theta }=e^{\theta \mathfrak{S}}=I_{3}+\left(
\sinh \theta \right) \mathfrak{S}-\left( 1-\cosh \theta \right) \mathfrak{S}%
^{2}$ & ,$\text{if }\mathbf{v}\text{ is a unit spacelike}$ \\ 
$\mathfrak{R}_{\mathbf{v}}^{\theta }=e^{-\theta \mathfrak{S}}=I_{3}-\theta 
\mathfrak{S}+\dfrac{\theta ^{2}\mathfrak{S}^{2}}{2!}$ & ,$\text{if }\mathbf{v%
}\text{ is a lightlike vector}$%
\end{tabular}%
\]%
gives a rotation on the quadric surface $\mathfrak{Q}$ where $\mathbf{v}=x%
\mathbf{i+}y\mathbf{j+}z\mathbf{k}$ is a rotation axis, and $I_{3}$ is the
identity matrix. These matrices are as follows, respectively.%
\begin{eqnarray}
&&%
\begin{bmatrix}
\begin{array}{c}
\left( 1-\cos \theta \right) \sigma _{11} \\ 
-\left( \sin \theta \right) \left( y\alpha _{1}+z\beta _{1}\right) +1%
\end{array}
& 
\begin{array}{c}
\left( \cos \theta -1\right) \sigma _{12} \\ 
+x\alpha _{1}\sin \theta -z\lambda _{1}\sin \theta%
\end{array}
& 
\begin{array}{c}
\left( \cos \theta -1\right) \sigma _{13} \\ 
+x\beta _{1}\sin \theta +y\lambda _{1}\sin \theta%
\end{array}
\\ 
\begin{array}{c}
\sigma _{21}\left( \cos \theta -1\right) \\ 
-\left( \sin \theta \right) \left( y\alpha _{2}+z\beta _{2}\right)%
\end{array}
& 
\begin{array}{c}
\left( 1-\cos \theta \right) \sigma _{22} \\ 
+x\alpha _{2}\sin \theta +z\beta _{1}\sin \theta +1%
\end{array}
& 
\begin{array}{c}
\sigma _{23}\left( \cos \theta -1\right) \\ 
+x\beta _{2}\sin \theta -y\beta _{1}\sin \theta%
\end{array}
\\ 
\begin{array}{c}
\left( \cos \theta -1\right) \sigma _{31} \\ 
+\left( z\alpha _{2}-y\alpha _{3}\right) \sin \theta%
\end{array}
& 
\begin{array}{c}
\left( \cos \theta -1\right) \sigma _{32} \\ 
+\left( x\alpha _{3}-z\alpha _{1}\right) \sin \theta%
\end{array}
& 
\begin{array}{c}
\sigma _{33}\left( 1-\cos \theta \right) \\ 
+\left( y\alpha _{1}-x\alpha _{2}\right) \sin \theta +1%
\end{array}%
\end{bmatrix}
\label{complex} \\
&&%
\begin{bmatrix}
\begin{array}{c}
\sigma _{11}\left( \cosh \theta -1\right) \\ 
-\left( \sinh \theta \right) \left( y\alpha _{1}+z\beta _{1}\right) +1%
\end{array}
& 
\begin{array}{c}
\sigma _{12}\left( 1-\cosh \theta \right) \\ 
+x\alpha _{1}\sinh \theta -z\lambda _{1}\sinh \theta%
\end{array}
& 
\begin{array}{c}
\sigma _{13}\left( 1-\cosh \theta \right) \\ 
+\left( x\beta _{1}+y\lambda _{1}\right) \sinh \theta%
\end{array}
\\ 
\begin{array}{c}
\sigma _{21}\left( 1-\cosh \theta \right) \\ 
-\left( y\alpha _{2}+z\beta _{2}\right) \sinh \theta%
\end{array}
& 
\begin{array}{c}
\sigma _{22}\left( \cosh \theta -1\right) \\ 
+\left( x\alpha _{2}+z\beta _{1}\right) \sinh \theta +1%
\end{array}
& 
\begin{array}{c}
\sigma _{23}\left( 1-\cosh \theta \right) \\ 
+\left( x\beta _{2}-y\beta _{1}\right) \sinh \theta%
\end{array}
\\ 
\begin{array}{c}
\sigma _{31}\left( 1-\cosh \theta \right) \\ 
+\left( z\alpha _{2}-y\alpha _{3}\right) \sinh \theta%
\end{array}
& 
\begin{array}{c}
\sigma _{32}\left( 1-\cosh \theta \right) \\ 
+\left( x\alpha _{3}-z\alpha _{1}\right) \sinh \theta%
\end{array}
& 
\begin{array}{c}
\sigma _{33}\left( \cosh \theta -1\right) \\ 
+\left( y\alpha _{1}-x\alpha _{2}\right) \sinh \theta +1%
\end{array}%
\end{bmatrix}
\label{hiperbolik} \\
&&%
\begin{bmatrix}
\begin{array}{c}
\frac{1}{2}\theta ^{2}\sigma _{11} \\ 
+\left( y\alpha _{1}+z\beta _{1}\right) \theta +1%
\end{array}
& 
\begin{array}{c}
-\frac{1}{2}\theta ^{2}\sigma _{12} \\ 
+\left( z\lambda _{1}-x\alpha _{1}\right) \theta%
\end{array}
& 
\begin{array}{c}
-\frac{1}{2}\theta ^{2}\sigma _{13} \\ 
-y\theta \lambda _{1}-x\theta \beta _{1}%
\end{array}
\\ 
\begin{array}{c}
-\frac{1}{2}\theta ^{2}\sigma _{21} \\ 
+\theta \left( y\alpha _{2}+z\beta _{2}\right)%
\end{array}
& 
\begin{array}{c}
\frac{1}{2}\theta ^{2}\sigma _{22} \\ 
-\left( x\alpha _{2}+z\beta _{1}\right) \theta +1%
\end{array}
& 
\begin{array}{c}
-\frac{1}{2}\theta ^{2}\sigma _{23} \\ 
+y\theta \beta _{1}-x\theta \beta _{2}%
\end{array}
\\ 
\begin{array}{c}
-\frac{1}{2}\theta ^{2}\sigma _{31} \\ 
+\left( y\alpha _{3}-z\alpha _{2}\right) \theta%
\end{array}
& 
\begin{array}{c}
-\frac{1}{2}\theta ^{2}\sigma _{32} \\ 
+\left( z\alpha _{1}-x\alpha _{3}\right) \theta%
\end{array}
& 
\begin{array}{c}
\frac{1}{2}\theta ^{2}\sigma _{33} \\ 
+x\theta \alpha _{2}-y\theta \alpha _{1}+1%
\end{array}%
\end{bmatrix}
\label{dual}
\end{eqnarray}%
where the coefficients $\sigma _{ij}$ is following in the table:%
\[
\begin{tabular}{|l|l|}
\hline
$\sigma _{11}=\left( By^{2}+Cz^{2}+Dxy+Exz+2Fyz\right) $ & $\sigma
_{23}=\left( Fy^{2}+Cyz+Exy\right) $ \\ \hline
$\sigma _{12}=\left( Dx^{2}+Bxy+Fxz\right) $ & $\sigma _{31}=\left(
Ez^{2}+Dyz+Axz\right) $ \\ \hline
$\sigma _{13}=\left( Ex^{2}+Fxy+Cxz\right) $ & $\sigma _{32}=\left(
Fz^{2}+Dxz+Byz\right) $ \\ \hline
$\sigma _{21}=\left( Dy^{2}+Eyz+Axy\right) $ & $\sigma _{33}=\left(
Ax^{2}+By^{2}+2xyD+xzE+Fyz\right) $ \\ \hline
$\sigma _{22}=\left( Ax^{2}+Cz^{2}+xyD+2xzE+Fyz\right) $ &  \\ \hline
\end{tabular}%
\]
\end{theorem}

\begin{proof}
Let's first find the characteristic polynomial of the matrix $\mathfrak{S}$
required for the Rodrigues rotation formula and the powers of the matrix $%
\mathfrak{S}$. The characteristic polynomial of $\mathfrak{S}$ is%
\[
X^{3}-\mathcal{C}_{\mathbf{v}}\left( \mathbf{v}\right) X=0.
\]%
According to Cayley-Hamilton theorem, we have%
\[
\mathfrak{S}^{3}=\mathcal{C}_{\mathbf{v}}\left( \mathbf{v}\right) \mathfrak{%
S.}
\]%
The relation between the powers of $\mathfrak{S}$ are%
\[
\mathfrak{S}^{4}=\mathcal{C}_{\mathbf{v}}\left( \mathbf{v}\right) \mathfrak{S%
}^{2}\text{, }\mathfrak{S}^{5}=\mathcal{C}_{\mathbf{v}}\left( \mathbf{v}%
\right) \mathfrak{S}^{3}\text{, }\mathfrak{S}^{6}=\left( \mathcal{C}_{%
\mathbf{v}}\left( \mathbf{v}\right) \right) ^{2}\mathfrak{S}^{2}\text{, }%
\mathfrak{S}^{7}=\left( \mathcal{C}_{\mathbf{v}}\left( \mathbf{v}\right)
\right) ^{3}\mathfrak{S}\text{ ...}
\]%
and so we have%
\begin{equation}
\begin{tabular}{lll}
$\mathfrak{S}^{2n}=\left( \mathcal{C}_{\mathbf{v}}\left( \mathbf{v}\right)
\right) ^{n-1}\mathfrak{S}^{2}$ & ,$\mathfrak{S}^{2n-1}=\left( \mathcal{C}_{%
\mathbf{v}}\left( \mathbf{v}\right) \right) ^{n-1}\mathfrak{S}$ & ,$\text{if 
}\mathbf{v}\text{ is a unit ellipsoid or timelike}$ \\ 
$\mathfrak{S}^{2n}=\mathfrak{S}^{2}$ & ,$\mathfrak{S}^{2n-1}=\mathfrak{S}$ & 
,$\text{if }\mathbf{v}\text{ is a unit spacelike}$ \\ 
$\mathfrak{S}^{2n}=0$ & ,$\mathfrak{S}^{2n-1}=0$ for $n>2$ & ,$\text{if }%
\mathbf{v}\text{ is a lightlike}$%
\end{tabular}
\label{M--usleri quad}
\end{equation}%
For unit ellipsoid or timelike vector $\mathbf{v}$, the rotation matrix is%
\begin{eqnarray*}
\mathfrak{R}_{\mathbf{v}}^{\theta } &=&e^{\theta \mathfrak{S}}=I_{3}+\theta 
\mathfrak{S}+\dfrac{\theta ^{2}\mathfrak{S}^{2}}{2!}+\dfrac{\theta ^{3}%
\mathfrak{S}^{3}}{3!}+\dfrac{\theta ^{4}\mathfrak{S}^{4}}{4!}+... \\
&=&I_{3}+\theta \mathfrak{S}+\dfrac{\theta ^{2}\mathfrak{S}^{2}}{2!}-\dfrac{%
\theta ^{3}\mathfrak{S}}{3!}-\dfrac{\theta ^{4}\mathfrak{S}^{2}}{4!}+... \\
&=&I_{3}+\left( \theta -\dfrac{\theta ^{3}}{3!}+\dfrac{\theta ^{5}}{5!}%
-...\right) \mathfrak{S}+\left( \dfrac{\theta ^{2}}{2!}-\dfrac{\theta ^{4}}{%
4!}+\dfrac{\theta ^{6}}{6!}-...\right) \mathfrak{S}^{2} \\
\mathfrak{R}_{\mathbf{v}}^{\theta } &=&I_{3}+\left( \sin \theta \right) 
\mathfrak{S}+\left( 1-\cos \theta \right) \mathfrak{S}^{2}\text{.}
\end{eqnarray*}%
For unit spacelike vector $\mathbf{v}$, the rotation matrix is%
\begin{eqnarray*}
\mathfrak{R}_{\mathbf{q}}^{\theta } &=&e^{\theta \mathfrak{S}}=I_{3}+\theta 
\mathfrak{S}+\dfrac{\theta ^{2}\mathfrak{S}^{2}}{2!}+\dfrac{\theta ^{3}%
\mathfrak{S}^{3}}{3!}+\dfrac{\theta ^{4}\mathfrak{S}^{4}}{4!}+... \\
&=&I_{3}+\theta \mathfrak{S}+\dfrac{\theta ^{2}\mathfrak{S}^{2}}{2!}+\dfrac{%
\theta ^{3}\mathfrak{S}}{3!}+\dfrac{\theta ^{4}\mathfrak{S}^{2}}{4!}+... \\
&=&I_{3}+\left( \theta +\dfrac{\theta ^{3}}{3!}+\dfrac{\theta ^{5}}{5!}%
+...\right) \mathfrak{S}+\left( \dfrac{\theta ^{2}}{2!}+\dfrac{\theta ^{4}}{%
4!}+\dfrac{\theta ^{6}}{6!}+...\right) \mathfrak{S}^{2} \\
\mathfrak{R}_{\mathbf{q}}^{\theta } &=&I_{3}+\left( \sinh \theta \right) 
\mathfrak{S}-\left( 1-\cosh \theta \right) \mathfrak{S}^{2}.
\end{eqnarray*}%
For unit lightlike vector $\mathbf{v}$, the rotation matrix is%
\[
\mathfrak{R}_{\mathbf{v}}^{\theta }=e^{-\theta \mathfrak{S}}=I_{3}-\theta 
\mathfrak{S}+\dfrac{\theta ^{2}\mathfrak{S}^{2}}{2!}\text{.}
\]%
Extending these equations gives the quadratic rotation matrices (\ref%
{complex}), (\ref{hiperbolik}), (\ref{dual}).
\end{proof}

\subsection{Cayley Representation for Quadratic numbers}

For the skew-symmetric matrix $\mathfrak{S}$ of the form (\ref%
{skewsymquadric}), $\left( I_{3}+\mathfrak{S}\right) $ have an inverse, and
its Cayley map is%
\[
\mathcal{R}_{\mathbf{v}}\left( \mathfrak{S}\right) =\left( I_{3}-\mathfrak{S}%
\right) \left( I_{3}+\mathfrak{S}\right) ^{-1}=\left( I_{3}+\mathfrak{S}%
\right) ^{-1}\left( I_{3}-\mathfrak{S}\right)
\]%
where the unit vector $\mathbf{v}$ forming the entires of $\mathfrak{S}$ is
not a unit spacelike quadratic vector \cite{erdogdu2015},\cite{nevsovic2016}%
, \cite{ozdemir2016}.

\begin{theorem}
Let $\mathfrak{S}$ is a skew-symmetric matrix with $\mathbf{v}=x\mathbf{i+}y%
\mathbf{j+}z\mathbf{k}$ in the form (\ref{skewsymquadric}). If $\mathbf{v}$
is not a unit spacelike vector, then%
\[
\mathcal{R}_{\mathbf{v}}\left( \mathfrak{S}\right) =\left( I_{3}-\mathfrak{S}%
\right) \left( I_{3}+\mathfrak{S}\right) ^{-1}
\]%
is a quadratic rotation matrix where $\mathbf{v}$ is the rotation matrix.
So, we get the following rotation matrix:%
\[
\mathcal{R}_{\mathbf{v}}=\rho \mathcal{M}_{\mathbf{v}}
\]%
where $\rho =\frac{1}{\Delta \left\langle \mathbf{v,v}\right\rangle _{%
\mathfrak{Q}}-1}$ and%
\[
\mathcal{M=}%
\begin{bmatrix}
\begin{array}{c}
Ax^{2}-By^{2}-Cz^{2} \\ 
-2Fyz-2z\beta _{1} \\ 
-2y\alpha _{1}-1%
\end{array}
& 2\left( 
\begin{array}{c}
Fxz+Bxy+Dx^{2} \\ 
+x\alpha _{1}-z\lambda _{1}%
\end{array}%
\right) & 2\left( 
\begin{array}{c}
Ex^{2}+Fxy+Cxz \\ 
+x\beta _{1}+y\lambda _{1}%
\end{array}%
\right) \\ 
2\left( 
\begin{array}{c}
Axy+Dy^{2}+Eyz \\ 
-z\beta _{2}-y\alpha _{2}%
\end{array}%
\right) & 
\begin{array}{c}
-Ax^{2}+By^{2}-Cz^{2} \\ 
-2Exz+2z\beta _{1} \\ 
+2x\alpha _{2}-1%
\end{array}
& 2\left( 
\begin{array}{c}
Fy^{2}+Exy+Cyz \\ 
+x\beta _{2}-y\beta _{1}%
\end{array}%
\right) \\ 
2\left( 
\begin{array}{c}
Axz+Ez^{2}+Dyz \\ 
+z\alpha _{2}-y\alpha _{3}%
\end{array}%
\right) & 2\left( 
\begin{array}{c}
Fz^{2}+Dxz+Byz \\ 
+x\alpha _{3}-z\alpha _{1}%
\end{array}%
\right) & 
\begin{array}{c}
-Ax^{2}-By^{2}+Cz^{2} \\ 
-2xyD+2y\alpha _{1} \\ 
-2x\alpha _{2}-1%
\end{array}%
\end{bmatrix}%
.
\]
\end{theorem}

\begin{proof}
We have%
\[
\left( I+\mathfrak{S}\right) ^{T}\mathfrak{M}_{\mathfrak{Q}}=\mathfrak{M}_{%
\mathfrak{Q}}\left( I-\mathfrak{S}\right) \text{ and }\left( I-\mathfrak{S}%
\right) ^{T}\mathfrak{M}_{\mathfrak{Q}}=\mathfrak{M}_{\mathfrak{Q}}\left( I+%
\mathfrak{S}\right) \text{.}
\]

Using these equalities, it can be find that%
\[
\left( \mathcal{R}_{\mathbf{v}}\right) ^{T}\mathfrak{M}_{\mathfrak{Q}}\left( 
\mathcal{R}_{\mathbf{v}}\right) =\left( \left( I+\mathfrak{S}\right) \left(
I-\mathfrak{S}\right) ^{-1}\right) ^{T}\mathfrak{M}_{\mathfrak{Q}}\left( I+%
\mathfrak{S}\right) \left( I-\mathfrak{S}\right) ^{-1}=\mathfrak{M}_{%
\mathfrak{Q}}\text{.}
\]%
Also, we have $\det \mathcal{R}_{\mathbf{v}}=1$ because of $\det \left( I+%
\mathfrak{M}_{\mathfrak{Q}}\right) =1-\left\langle \mathbf{v,v}\right\rangle
_{\mathfrak{Q}}$ and $\det \left( I-\mathfrak{M}_{\mathfrak{Q}}\right) ^{-1}=%
\dfrac{1}{1-\left\langle \mathbf{v,v}\right\rangle _{\mathfrak{Q}}}.$ That
is, $\mathcal{R}_{\mathbf{v}}$ is a quadratic rotation matrix.$\ $Also, the
eigenvalue the matrix $\mathcal{R}_{\mathbf{v}}$ are%
\[
\mathfrak{e}_{1},\mathfrak{e}_{2}=\pm \frac{\left( \Delta \left\langle 
\mathbf{v,v}\right\rangle _{\mathfrak{Q}}+2\sqrt{\Delta \left\langle \mathbf{%
v,v}\right\rangle _{\mathfrak{Q}}}+1\right) }{1-\Delta \left\langle \mathbf{%
v,v}\right\rangle _{\mathfrak{Q}}},\text{ and }\mathfrak{e}_{3}=1.
\]%
And the eigenvector corresponding to $1$ is $\mathbf{v}$. The rotation angle 
$\theta $ is given by%
\[
\begin{tabular}{ll}
$\tan \theta =\frac{2\left\Vert \mathbf{v}\right\Vert _{\mathfrak{Q}}}{%
1-\left\Vert \mathbf{v}\right\Vert _{\mathfrak{Q}}^{2}}$ & ,$\text{if }%
\mathbf{v}\text{ is a unit ellipsoid or timelike}$ \\ 
$\cosh \theta =\dfrac{1+\left\langle \mathbf{v,v}\right\rangle _{\mathfrak{Q}%
}}{1-\left\langle \mathbf{v,v}\right\rangle _{\mathfrak{Q}}}$ & ,$\text{if }%
\mathbf{v}\text{ is a unit spacelike}$ \\ 
$\theta =1$ & ,$\text{if }\mathbf{v}\text{ is a lightlike}$%
\end{tabular}%
.
\]
\end{proof}

\begin{remark}
When constructing the number system of any quadratic form $\mathfrak{Q}$,
the following equations between the coefficients of the quadratic form and
the coefficients $\alpha _{1},\alpha _{2},\alpha _{3},\beta _{1},\beta _{2},$
and $\lambda _{1}$ of the number system can be used.%
\[
\begin{tabular}{|l|l|l|l|l|l|l|}
\hline
Number & $\alpha _{1}$ & $\alpha _{2}$ & $\alpha _{3}$ & $\beta _{1}$ & $%
\beta _{2}$ & $\lambda _{1}$ \\ \hline
Value & $-\dfrac{M_{13}}{\Delta \left\vert \Gamma \right\vert }$ & $\dfrac{%
M_{23}}{\Delta \left\vert \Gamma \right\vert }$ & $-\dfrac{M_{33}}{\Delta
\left\vert \Gamma \right\vert }$ & $-\dfrac{M_{12}}{\Delta \left\vert \Gamma
\right\vert }$ & $\dfrac{M_{22}}{\Delta \left\vert \Gamma \right\vert }$ & $-%
\dfrac{M_{11}}{\Delta \left\vert \Gamma \right\vert }$ \\ \hline
\end{tabular}%
\]%
where $M_{ij}$ is minors of the matrix $\Delta M_{\mathfrak{Q}}$ and $%
\left\vert \Gamma \right\vert =\sqrt{\left\vert \det \Delta M_{\mathfrak{Q}%
}\right\vert }$. Note that if some of the quadratic coefficients are zero,
these numbers $\alpha _{1},\alpha _{2},\alpha _{3},\beta _{1},\beta _{2},$
and $\lambda _{1}$ can be obtained differently.
\end{remark}

Now let's give examples of how the rotation transformation is realized using
positive definite and indefinite quadratic numbers.

\begin{example}
Let's create a positive definite bilinear form with the following ellipsoid%
\[
6x^{2}+6xy+4xz+2y^{2}+4yz+3z^{2}=1
\]%
and exemplify some of the geometric results obtained throughout the article.
The coefficients of the bilinear form and quadratic number system
corresponding to this ellipsoid are as follows:%
\begin{equation}
\left\langle \mathbf{u},\mathbf{v}\right\rangle _{\mathfrak{EQ}}=%
\begin{bmatrix}
x \\ 
y \\ 
z%
\end{bmatrix}%
^{T}%
\begin{bmatrix}
6 & 3 & 2 \\ 
3 & 2 & 2 \\ 
2 & 2 & 3%
\end{bmatrix}%
\begin{bmatrix}
x \\ 
y \\ 
z%
\end{bmatrix}
\label{ellipsoidornek}
\end{equation}
\end{example}

\[
\begin{tabular}{|l|l|l|l|l|l|l|}
\hline
Number & $\alpha _{1}$ & $\alpha _{2}$ & $\alpha _{3}$ & $\beta _{1}$ & $%
\beta _{2}$ & $\lambda _{1}$ \\ \hline
Value & $2$ & $-6$ & $3$ & $5$ & $-14$ & $2$ \\ \hline
\end{tabular}%
\]

Let's take the vector $\mathbf{v}=\left( 0,0,1/\sqrt{3}\right) $ and the
point $A=\left( 0,1/\sqrt{2},0\right) $ over this ellipsoid. Let us take the
plane%
\[
\frac{2}{3}x+\frac{2}{3}y+z-\frac{1}{3}\sqrt{2}=0,
\]%
which is orthogonal to the vector $\mathbf{v}$ and passes through the point $%
A$. The conic formed by the intersection of this quadric surface (\ref%
{ellipsoidornek}) and the given plane and its center are as follows:%
\[
\begin{tabular}{ll}
$\frac{14}{3}x^{2}+\frac{10}{3}xy+\frac{2}{3}y^{2}+\frac{2}{3}=1$ & , $%
C=\left( 0,0,\frac{1}{3}\sqrt{2}\right) $%
\end{tabular}%
\]%
For any $\mathbf{u,v}$, the vector product on this quadratic form is as
follows:%
\[
\left\vert 
\begin{array}{ccc}
2\mathbf{i}-5\mathbf{j+}2\mathbf{k} & -5\mathbf{i}+14\mathbf{j}-6\mathbf{k}
& 2\mathbf{i}-6\mathbf{j+}3\mathbf{k} \\ 
u_{1} & u_{2} & u_{3} \\ 
v_{1} & v_{2} & v_{3}%
\end{array}%
\right\vert
\]%
Let's take 3 points $A\left( 0,1/\sqrt{2},0\right) ,C,B=\left( 1/\sqrt{2}%
,0,0\right) $ on the given plane and show that the normal of the plane is $%
\mathbf{v}$ with the help of the vector product.%
\[
CA\times CB=\left\vert 
\begin{array}{ccc}
2\mathbf{i}-5\mathbf{j+}2\mathbf{k} & -5\mathbf{i}+14\mathbf{j}-6\mathbf{k}
& 2\mathbf{i}-6\mathbf{j+}3\mathbf{k} \\ 
0 & 1/\sqrt{2} & -\frac{1}{3}\sqrt{2} \\ 
1/\sqrt{2} & 0 & -\frac{1}{3}\sqrt{2}%
\end{array}%
\right\vert =\left( 0,0,-1/6\right)
\]%
So, we get%
\[
\mathbf{v=}\left( OA\times OB\right) /\left\Vert OA\times OB\right\Vert .
\]

\end{document}